\documentclass{amsart}
\usepackage{hyperref}
\usepackage{latexsym}
\usepackage{todonotes}

\newtheorem{theorem}{Theorem}[section]

\newtheorem{proposition}[theorem]{Proposition}
\newtheorem{lemma}[theorem]{Lemma}
\newtheorem{conjecture}[theorem]{Conjecture}
\theoremstyle{definition}

\newtheorem{remark}[theorem]{Remark}
\newtheorem{question}[theorem]{Question}

\newcommand{\Z}{\mathbb{Z}}

\newcommand{\C}{\mathbb{C}}

\renewcommand\a{\mathtt{0}}
\renewcommand\b{\mathtt{1}}
\renewcommand\c{\mathtt{2}}
\renewcommand\d{\mathtt{3}}

\newcommand{\GL}{\mathrm{GL}}

\makeatletter
\@namedef{subjclassname@2020}{\textup{2020} Mathematics Subject Classification}
\makeatother

\begin{document}
\title
[A {$q$}-analog of the {M}arkoff injectivity conjecture holds]
{A {$q$}-analog of the {M}arkoff\\injectivity conjecture holds}

\author[S.~Labb\'e]{S\'ebastien Labb\'e}
\address[S.~Labb\'e]{Univ. Bordeaux, CNRS, Bordeaux INP, LaBRI, UMR 5800, F-33400, Talence, France}
\email{sebastien.labbe@labri.fr}
\urladdr{http://www.slabbe.org/}

\author[M.~Lapointe]{M\'elodie Lapointe}
\address[M.~Lapointe]{Universit\'e de Moncton,
D\'epartement de math\'ematiques et de statistique,
18 avenue Antonine-Maillet,
Moncton NB E1A 3E9, Canada}
\email{melodie.lapointe@umoncton.ca}
\urladdr{http://lapointemelodie.github.io/}

\author[W. Steiner]{Wolfgang Steiner}
\address[W.~Steiner]{Universit\'e Paris Cit\'e, CNRS, IRIF, F--75006 Paris, France}
\email{steiner@irif.fr}
\urladdr{https://www.irif.fr/~steiner}

\keywords{Markoff number; Christoffel word; $q$-analog}
\subjclass[2020]{Primary 11J06; Secondary 68R15 \and 05A30}

\begin{abstract}
 The elements of Markoff triples are given by coefficients in certain matrix products defined by Christoffel words, and the Markoff injectivity conjecture, a long-standing open problem (also known as the uniqueness conjecture), is then equivalent to injectivity on Christoffel words. A $q$-analog of these matrix products has been proposed recently, and we prove that injectivity on Christoffel words holds for this $q$-analog. The proof is based on the evaluation at $q = \exp(i\pi/3)$. Other roots of unity provide some information on the original problem, which corresponds to the case $q=1$. We also extend the problem to arbitrary words and provide a large family of pairs of words where injectivity does not hold.
\end{abstract}

\maketitle

\section{Introduction}

Christoffel words are words over the alphabet $\{\a,\b\}$ that can be defined
recursively as follows: $\a$, $\b$ and $\a\b$ are Christoffel words and if
$u,v,uv\in\{\a,\b\}^*$ are Christoffel words then $uuv$ and $uvv$ are Christoffel
words \cite{MR2464862}. The shortest Christoffel words are:
\[
    \a,\b,
    \a\b,
    \a\a\b,\a\b\b,
    \a\a\a\b,\a\a\b\a\b,\a\b\a\b\b,\a\b\b\b,
    \a\a\a\a\b,\a\a\a\b\a\a\b,\a\a\b\a\a\b\a\b,\a\a\b\a\b\a\b,
    \ldots
\]
Note that these are also named \emph{lower} Christoffel words.

A Markoff triple is a positive solution of the Diophantine equation
$x^2+y^2+z^2=3xyz$ \cite{MR1510073,M1879}.
Markoff triples can be defined recursively as follows:
$(1,1,1)$, $(1,2,1)$ and $(1,5,2)$ are Markoff triples and if
$(x,y,z)$ is a Markoff triple with $y\geq x$ and $y\geq z$, then
$(x,3xy-z,y)$ and $(y,3yz-x,z)$ are Markoff triples.
A list of small Markoff numbers (elements of a Markoff triple) is
\[
    1,2,5,13,29,34,89,169,194,233,433,610,985,1325,1597,2897,4181,
    \ldots
\]
referenced as sequence \href{https://oeis.org/A002559}{A002559} in OEIS \cite{OEISA005132}.

It is known that each Markoff number can be expressed in terms of a Christoffel
word.
More precisely, 
let $\mu$ be the monoid homomorphism $\{\a,\b\}^* \rightarrow \GL_2(\mathbb{Z})$ defined by 
\[
    \mu(\a) = \begin{pmatrix} 2 & 1 \\ 1 & 1 \end{pmatrix}
\quad
\text{ and }
\quad
    \mu(\b) = \begin{pmatrix} 5 & 2 \\ 2 & 1 \end{pmatrix}.
\] 
Each Markoff number is equal to $\mu(w)_{12}$ for some Christoffel word $w$ 
\cite{MR2534916}, 
where
$M_{12}$ denotes
the element above the diagonal in a matrix 
$M=\left(\begin{smallmatrix}M_{11}&M_{12}\\M_{21}&M_{22}\end{smallmatrix}\right)\in\GL_2(\mathbb{Z})$.

For example, the Markoff number 194 is associated with the Christoffel word $\a\a\b\a\b$
as it is the entry at position $(1,2)$ in the matrix
\begin{align*}
\mu(\a\a\b\a\b)
&=
\begin{pmatrix} 2 & 1 \\ 1 & 1 \end{pmatrix}
\begin{pmatrix} 2 & 1 \\ 1 & 1 \end{pmatrix}
\begin{pmatrix} 5 & 2 \\ 2 & 1 \end{pmatrix}
\begin{pmatrix} 2 & 1 \\ 1 & 1 \end{pmatrix}
\begin{pmatrix} 5 & 2 \\ 2 & 1 \end{pmatrix}
=
\begin{pmatrix} 463 & 194 \\ 284 & 119 \end{pmatrix}.
\end{align*}
Whether the map $w\mapsto\mu(w)_{12}$ provides a bijection between Christoffel
words and Markoff numbers is a question (stated differently in \cite{F1913})
that has remained open for more than 100 years \cite{MR3098784}.
The conjecture can be expressed in terms of the
injectivity of the map $w\mapsto\mu(w)_{12}$ \cite[\S 3.3]{MR3887697}.

\begin{conjecture}[Markoff Injectivity Conjecture]
    The map $w\mapsto\mu(w)_{12}$ is injective on the set of Christoffel words.
\end{conjecture}

In \cite{MR4073883}, a $q$-analog of rational numbers and of
continued fractions were introduced.
This was the inspiration for several advances
\cite{MR4266256,Ovsienko2021,MR4407768,MR4462940,MR4499341}
and among them a $q$-analog of Markoff triples \cite{kogiso_q-deformations_2020}.
A $q$-analog of the matrices $\mu(\a)$ and $\mu(\b)$ was proposed in \cite{MR4265544}, which
in terms of
\[
L_q = \begin{pmatrix}q&0\\q&1\end{pmatrix}
    \quad
    \text{ and }
    \quad
R_q = \begin{pmatrix}q&1\\0&1\end{pmatrix},
\]
can be written as
\begin{align*}
    \mu_q(\a)
    &=R_qL_q
   =
\begin{pmatrix}
    q + q^{2} & 1 \\
    q & 1
\end{pmatrix},\\
    \mu_q(\b)
    &=R_qR_qL_qL_q
   =
\begin{pmatrix}
    q + 2q^2+q^3+q^4 & 1 + q \\
    q + q^{2} & 1
\end{pmatrix}.
\end{align*}
It extends to a morphism of monoids $\mu_q:\{\a,\b\}^*\to\GL_2(\Z[q^{\pm 1}])$.
This $q$-analog satisfies that $\mu_1(w)=\mu(w)$ for every $w\in\{\a,\b\}^*$.
Thus if $w$ is a Christoffel word, then
the entry above the diagonal $\mu_q(w)_{12}$ is a polynomial
of indeterminate $q$ 
with nonnegative integer coefficients such that 
it is a Markoff number when evaluated at $q=1$.
For example,
\[
\mu_q(\a\a\b\a\b)_{12}
= 1 + 4 q + 10 q^{2} + 18 q^{3} + 27 q^{4} + 33 q^{5} + 33 q^{6} + 29 q^{7} +
    21 q^{8} + 12 q^{9} + 5 q^{10} + q^{11}
\]
which, when evaluated at $q=1$, is equal to
\[
\mu_1(\a\a\b\a\b)_{12}
= 1 + 4 + 10 + 18 + 27 + 33 + 33 + 29 + 21 + 12 + 5 + 1 = 194.
\]

In \cite{MR4405998}, a $q$-analog of the Markoff Injectivity Conjecture was considered based on
the map $w\mapsto\mu_q(w)_{12}$. It was proved that the map is injective over
the language of any fixed Christoffel word,
extending a result proved when $q=1$ \cite{MR4281387}.
In this work, we go one step further and prove
a $q$-analog of the Markoff Injectivity Conjecture.

\begin{theorem}\label{thm:main}
    The map $w\mapsto\mu_q(w)_{12}$ is injective on the set of Christoffel words.
\end{theorem}

Theorem~\ref{thm:main} is proved in Section~\ref{sec:proof-thm-main}.
In Section~\ref{sec:muq-not-injective},
we give examples where the map $w\mapsto\mu_q(w)_{12}$ is not injective
when considered on the language $\{\a,\b\}^*$.

\subsection*{Acknowledgments}
This work was supported by the Agence Nationale de la Recherche through the
project Codys (ANR-18-CE40-0007) and IZES (ANR-22-CE40-0011). The second author
acknowledges the support of the Natural Sciences and Engineering Research
Council of Canada (NSERC), [funding reference number BP–545242–2020] and the
support of the Fonds de Recherche du Qu\'ebec en Science et Technologies.
We are thankful to the anonymous referee for their valuable comments.

\section{Proof of Theorem~\ref{thm:main}}\label{sec:proof-thm-main}

The main idea of this section is to evaluate the polynomial 
$\mu_q(w)_{12}$ at primitive root of unity $\zeta_k = \exp(2\pi i/k)$, in
particular when $k=6$.  

First, we observe that when $w\in\{\a,\b\}^*$, the matrix $\mu_{\zeta_6}(w)$
can be expressed in terms of $\zeta_6$, the length $|w|$ of $w$ and the number
$|w|_\b$ of occurrences of~$\b$ in $w$.

\begin{lemma}\label{lem:mu_q_evaluated_at_q6}
    For every $w\in\{\a,\b\}^*$,
    we have 
    \begin{equation}\label{e:mu-q6}
        \mu_{\zeta_6}(w)
        = \zeta_6^{|w|+|w|_\b}
        \left[
        \begin{pmatrix}
            |w|   & \hspace{-.75em} -|w|{-}|w|_\b \\
            -|w|_\b & \hspace{-.75em} -|w|
        \end{pmatrix}\zeta_6 + 
        \begin{pmatrix}
            |w|_\b  & \hspace{-.25em} |w| \\
            |w| {+} |w|_\b & \hspace{-.25em} {-}|w|_\b
        \end{pmatrix}
        +
        \begin{pmatrix}
            1 & \hspace{-.25em}  0\\
            0 & \hspace{-.25em}  1
        \end{pmatrix}
        \right].
    \end{equation}
\end{lemma}

\begin{proof}
    The proof is done by recurrence on the length of $w$.
    We have $\mu_{\zeta_6}(\varepsilon)=\left(\begin{smallmatrix}1&0\\0&1\end{smallmatrix}\right)$.
    Thus, the formula works for $w=\varepsilon$.
If \eqref{e:mu-q6} holds for~$w$, then we have
\[
\begin{split}
\mu_{\zeta_6}(w\a) 
& = \zeta_6^{|w|+|w|_\b} 
    \begin{pmatrix}
        |w|_\b + 1 + |w|\, \zeta_6 &
        |w|-(|w|{+}|w|_\b)\, \zeta_6  \\ 
        |w| + |w|_\b-|w|_\b\, \zeta_6 &
        1  - |w|_\b - |w|\, \zeta_6
    \end{pmatrix} 
    \zeta_6
    \begin{pmatrix}1{+}\zeta_6& 1{-}\zeta_6 \\ 1 & 1{-} \zeta_6 \end{pmatrix} \\
& = \zeta_6^{|w|+|w|_\b+1} 
    \begin{pmatrix}
        |w|_\b + 1 + (|w|{+}1)\, \zeta_6& 
        |w|+1-(|w|{+}|w|_\b{+}1)\, \zeta_6  \\
        |w| + |w|_\b +1 -|w|_\b\, \zeta_6 & 
        1 - |w|_\b-(|w|{+}1)\, \zeta_6 
    \end{pmatrix},\\
\mu_{\zeta_6}(w\b) 
& = \zeta_6^{|w|+|w|_\b} 
    \begin{pmatrix}
        |w|_\b + 1 + |w|\, \zeta_6 &
        |w|-(|w|{+}|w|_\b)\, \zeta_6  \\ 
        |w| + |w|_\b-|w|_\b\, \zeta_6 &
        1  - |w|_\b - |w|\, \zeta_6
    \end{pmatrix} 
    \zeta_6^2
    \begin{pmatrix}2{+}\zeta_6 & 1{-}2\zeta_6 \\ 2{-}\zeta_6 & - \zeta_6 \end{pmatrix} \\
& = \zeta_6^{|w|+|w|_\b+2} 
    \begin{pmatrix}
        |w|_\b + 2 + (|w|{+}1)\, \zeta_6 & 
        |w|+1 -(|w|{+}|w|_\b{+}2)\, \zeta_6 \\
        |w| + |w|_\b +2 -(|w|_\b{+}1)\, \zeta_6 & 
         - |w|_\b - (|w|{+}1)\, \zeta_6
    \end{pmatrix},
\end{split}
\]
hence \eqref{e:mu-q6} holds for $w\a$ and~$w\b$. 
\end{proof}

% The above proof can be checked with sage:
%
% sage: K.<zeta6> = CyclotomicField(6)
% sage: U.<w,w1> = K[]
% sage: U
% Multivariate Polynomial Ring in w, w1 over Cyclotomic Field of order 6 and degree 2
% sage: A = matrix(U,2,(w, -w-w1,-w1,-w))
% sage: B = matrix(U,2,(w1+1,w,w+w1,1-w1))
% sage: A*zeta6+B
% [            zeta6*w + w1 + 1 (-zeta6 + 1)*w + (-zeta6)*w1]
% [         w + (-zeta6 + 1)*w1          (-zeta6)*w - w1 + 1]
% sage: (A*zeta6+B) * matrix(2,(zeta6+1, -zeta6+1, 1, -zeta6+1))
% [                 zeta6*w + w1 + (zeta6 + 1) (-zeta6 + 1)*w + (-zeta6)*w1 + (-zeta6 + 1)]
% [                    w + (-zeta6 + 1)*w1 + 1              (-zeta6)*w - w1 + (-zeta6 + 1)]
% sage: (A*zeta6+B) * matrix(2,(zeta6+2, -2*zeta6+1, -zeta6+2,-zeta6))
% [                   zeta6*w + w1 + (zeta6 + 2) (-zeta6 + 1)*w + (-zeta6)*w1 + (-2*zeta6 + 1)]
% [           w + (-zeta6 + 1)*w1 + (-zeta6 + 2)                    (-zeta6)*w - w1 + (-zeta6)]

In particular, Equation~\eqref{e:mu-q6} implies that the entry above the diagonal is
\begin{equation}\label{eq:mu-q6-12}
    \mu_{\zeta_6}(w)_{12}
        = \zeta_6^{|w|+|w|_\b}\, \big(|w|-(|w|+|w|_\b) \zeta_6 \big)
        \in\C.
\end{equation}
The next result shows that when $w\in\{\a,\b\}^*\setminus\{\varepsilon\}$, the number
$\mu_{\zeta_6}(w)_{12}$ lies in one of the six cones of angle $\frac{\pi}{3}$ that
partition the complex plane according to the value of $|w|+|w|_\b$, see Figure~\ref{fig:6-cones}.

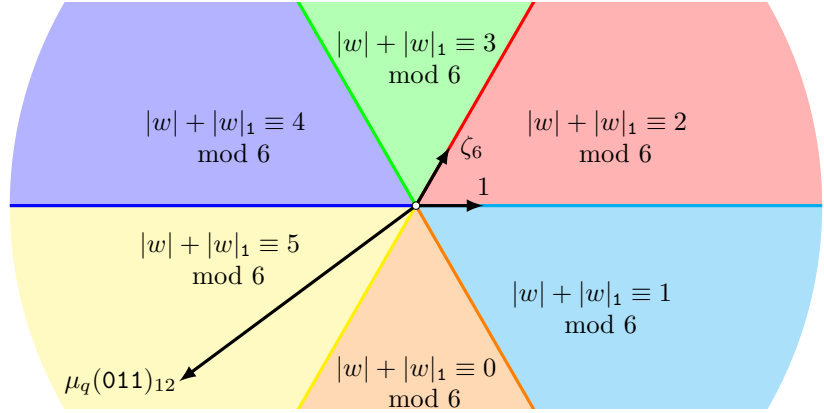
\begin{figure}
    \begin{tikzpicture}[scale=.9]
    \clip (-6,-3) rectangle (6,3);
    \fill[red!30]     (0,0) --  (60:6)  arc (60:0  :6);
    \fill[green!30]   (0,0) -- (120:6) arc (120:60 :6);
    \fill[blue!30]    (0,0) -- (180:6) arc (180:120:6);
    \fill[yellow!30]  (0,0) -- (240:6) arc (240:180:6);
    \fill[orange!30]  (0,0) -- (300:6) arc (300:240:6);
    \fill[cyan!30]    (0,0) -- (360:6) arc (360:300:6);
    \draw[very thick,red]     (0,0) --  (60:6);
    \draw[very thick,green]   (0,0) -- (120:6);
    \draw[very thick,blue]    (0,0) -- (180:6);
    \draw[very thick,yellow]  (0,0) -- (240:6);
    \draw[very thick,orange]  (0,0) -- (300:6);
    \draw[very thick,cyan]    (0,0) -- (360:6);
    \draw[very thick,-latex] (0,0) -- (60:1) node[right] {$\zeta_6$};
    \draw[very thick,-latex] (0,0) -- ( 0:1) node[above] {$1$};
    \node[align=center] at  (20:3)   {$|w|+|w|_\b \equiv 2$\\$\mod 6$};
    \node[align=center] at  (90:2.2) {$|w|+|w|_\b \equiv 3$\\$\mod 6$};
    \node[align=center] at (160:3)   {$|w|+|w|_\b \equiv 4$\\$\mod 6$};
    \node[align=center] at (195:3)   {$|w|+|w|_\b \equiv 5$\\$\mod 6$};
    \node[align=center] at (270:2.6) {$|w|+|w|_\b \equiv 0$\\$\mod 6$};
    \node[align=center] at (330:3)   {$|w|+|w|_\b \equiv 1$\\$\mod 6$};
    \draw[very thick,-latex] (0,0) -- (-3.5, -2.59807621135332) node[left=-3pt] {$\mu_q(\a\b\b)_{12}$};
    \node[circle,draw=black,fill=white,inner sep=1pt] at (0,0) {};
%sage: K.<zeta6> = CyclotomicField(6)
%sage: mu_q_12('abb')(q=zeta6)
%-3*zeta6 - 2
%sage: n(_)
%-3.50000000000000 - 2.59807621135332*I
\end{tikzpicture}
    \caption{
        A partition of the complex plane $\C\setminus\{0\}$ into six disjoint cones
        spanned by the vectors $\zeta_6^k$ and $\zeta_6^{k+1}$, $k\in\{0,1,2,3,4,5\}$.
        For every $w\in\{\a,\b\}^*\setminus\{\varepsilon\}$,
        $\mu_q(w)_{12}$ lies in the cone corresponding to 
        $|w|+|w|_\b \mod 6$.
        For instance,
        $\mu_q(\a\b\b)=\zeta_6^{5}(3-5\zeta_6)
                      =3\zeta_6^5-5$
        and $|w|+|w|_\b = |\a\b\b|+|\a\b\b|_\b = 3+2 \equiv 5 \mod 6$.
    }
    \label{fig:6-cones}
\end{figure}

\begin{lemma}\label{lem:6-cones-complex-plane}
    For every $w\in\{\a,\b\}^*\setminus\{\varepsilon\}$,
    we have 
    \[
        \mu_{\zeta_6}(w)_{12}\in 
        \left\{\rho\cdot e^{i\theta}
        \mid
        \rho>0,\,
        (|w|+|w|_\b+4)\tfrac{\pi}{3} 
        < \theta
        \leq (|w|+|w|_\b+5) \tfrac{\pi}{3}\right\}.
    \]
    Moreover, 
    $w=\varepsilon$ if and only if
        $\mu_{\zeta_6}(w)_{12}=0$.
\end{lemma}

\begin{proof}
    Let $w\in\{\a,\b\}^*\setminus\{\varepsilon\}$.
    Since $|w|+|w|_\b \geq |w|>0$, then
    observe that
    \[
        |w|-(|w|+|w|_\b) \zeta_6
        \in
        \big\{\rho\cdot e^{i\theta}
        \mid
        \rho>0,\,
        \tfrac{4\pi}{3} 
        < \theta
        \leq \tfrac{5\pi}{3}\big\}.
    \]
    Since $\zeta_6=e^{\frac{i\pi}{3}}$, 
    from Equation~\eqref{eq:mu-q6-12}, we have
    \begin{align*}
    \mu_{\zeta_6}(w)_{12}
        &= \zeta_6^{|w|+|w|_\b}\, \big(|w|-(|w|+|w|_\b) \zeta_6 \big)\\
        &\in e^{\frac{i\pi}{3}(|w|+|w|_\b)}\, 
        \big\{\rho\cdot e^{i\theta}
        \mid
        \rho>0,\,
        \tfrac{4\pi}{3} 
        < \theta
        \leq \tfrac{5\pi}{3}\big\}
        \\
        &= 
        \big\{\rho\cdot e^{i\theta}
        \mid
        \rho>0,\,
        \tfrac{(|w|+|w|_\b+4)\pi}{3} 
        < \theta
        \leq \tfrac{(|w|+|w|_\b+5)\pi}{3}\big\}.
    \end{align*}

    We have $\mu_{\zeta_6}(w)_{12} = 0$ if $w=\varepsilon$ and, from above, $\mu_{\zeta_6}(w)_{12} \ne 0$ if $w \in \{\a,\b\}^*\setminus\{\varepsilon\}$.
    Thus, if 
    $\mu_{\zeta_6}(w)_{12}=0$, then
    $w=\varepsilon$.
\end{proof}

The next result shows that we can recover the number of $\a$'s and $\b$'s
occurring in a word $w\in\{\a,\b\}^*$ from the polynomial $\mu_q(w)_{12}$
evaluated at $q=\zeta_6$.

\begin{proposition}\label{prop:recover-number-occ-letters}
    Let $w,w'\in\{\a,\b\}^*$.
    If $\mu_{\zeta_6}(w)_{12}=\mu_{\zeta_6}(w')_{12}$,
    then 
    $|w|_\a= |w'|_\a$
    and
    $|w|_\b= |w'|_\b$.
\end{proposition}

\begin{proof}
    If
    $\mu_{\zeta_6}(w)_{12}=\mu_{\zeta_6}(w')_{12}=0$,
    then from Lemma~\ref{lem:6-cones-complex-plane}, we have
    $w=\varepsilon=w'$, thus $|w|_\a=0=|w'|_\a$
    and
    $|w|_\b=0=|w'|_\b$.
    Now, assume that
    $\mu_{\zeta_6}(w)_{12}=\mu_{\zeta_6}(w')_{12}\neq 0$.
    From Lemma~\ref{lem:6-cones-complex-plane}, we have
    \begin{align*}
        \mu_{\zeta_6}(w)_{12}&\in 
        \big\{\rho\cdot e^{i\theta}
        \mid
        \rho>0,\,
        (|w|+|w|_\b+4)\tfrac{\pi}{3} 
        < \theta
        \leq (|w'|+|w'|_\b+5)\tfrac{\pi}{3}\big\},\\
        \mu_{\zeta_6}(w')_{12}&\in 
        \big\{\rho\cdot e^{i\theta}
        \mid
        \rho>0,\,
        (|w'|+|w'|_\b+4)\tfrac{\pi}{3} 
        < \theta
        \leq (|w'|+|w'|_\b+5)\tfrac{\pi}{3}\big\},
    \end{align*}
    which are two disjoint cones in the complex plane
    when
    $|w|+|w|_\b \not\equiv |w'|+|w'|_\b \mod 6$.
    Since
    $\mu_{\zeta_6}(w)_{12}=\mu_{\zeta_6}(w')_{12}$,
    the two cones must intersect and be equal. Therefore, we have
    $|w|+|w|_\b \equiv |w'|+|w'|_\b \mod 6$.
    From Lemma~\ref{lem:mu_q_evaluated_at_q6}, we have
    \begin{align*}
    |w'|-(|w'|+|w'|_\b) \zeta_6 
        &= \zeta_6^{-|w'|-|w'|_\b}\, \mu_{\zeta_6}(w')_{12} \\
        &= \zeta_6^{-|w'|-|w'|_\b}\, \mu_{\zeta_6}(w)_{12} \\
        &= \zeta_6^{-|w'|-|w'|_\b}\, \zeta_6^{|w|+|w|_\b}\, \big(|w|-(|w|+|w|_\b) \zeta_6 \big)\\
        &= |w|-(|w|+|w|_\b) \zeta_6.
    \end{align*}
    This implies that $|w'|=|w|$ and $|w'|+|w'|_\b= |w|+|w|_\b$.
    Then $|w|_\b= |w'|_\b$ and
    $|w|_\a=|w|-|w|_\b=|w'|-|w'|_\b= |w'|_\a$.
\end{proof}

We may now prove the main result. It is based on the isomorphism between the tree
of Christoffel words and the Stern--Brocot tree, a tree of positive rational numbers.
Indeed, the set of Christoffel words has the structure of a binary tree:
    if $u,v,uv\in\{\a,\b\}^*$ are Christoffel words, then $uuv$ and $uvv$ are
    the left and right children of the node $uv$ \cite[\S 3.2]{MR2464862}.
    The Christoffel tree is isomorphic to the Stern--Brocot tree via the map
    that associates to a vertex $w$ of the Christoffel tree the
    fraction $\frac{|w|_\b}{|w|_\a}$ \cite[Proposition 7.6]{MR2464862}.

\begin{proof}[Proof of Theorem~\ref{thm:main}]
    We want to show the injectivity of the map 
    $w\mapsto\mu_q(w)_{12}$ over the set of Christoffel words.
    Let $w,w'\in\{\a,\b\}^*$ be two Christoffel words
    such that
    $\mu_q(w)_{12}=\mu_q(w')_{12}$.
   In particular, we have
    $\mu_{\zeta_6}(w)_{12}=\mu_{\zeta_6}(w')_{12}$.
    From Proposition~\ref{prop:recover-number-occ-letters},
    $|w|_\a= |w'|_\a$ and $|w|_\b= |w'|_\b$.

    Suppose by contradiction that $w\neq w'$. This implies that the fraction
    $\frac{|w|_\b}{|w|_\a}=\frac{|w'|_\b}{|w'|_\a}$ appears twice in the
    Stern--Brocot tree. This is a contradiction because
    every positive rational number appears in the Stern--Brocot tree exactly
    once \cite[\S 4.5]{MR1397498}.
    Thus $w=w'$.
    Therefore the map $w\mapsto\mu_{\zeta_6}(w)_{12}$ is injective over the set of
    Christoffel words,
    and so is the map $w\mapsto\mu_q(w)_{12}$.
\end{proof}

% sage: kkk.<q6> = CyclotomicField(6)
% sage: q6
% q6
% sage: q6^2
% q6 - 1
% sage: q6^3
% -1
% sage: q6^4
% -q6
% sage: q6^5
% -q6 + 1

%Indeed, for $k \in \{2,3,4,5\}$, $\{\zeta_k^{-1} L_{\zeta_k},R_{\zeta_k}\}^*$ is a group of size $6$ (for $k=2$), $24$ (for $k=3$), $96$ (for $k=4$), $600$ (for $k=5$). 

\begin{remark}
The monoid generated by $\mu_{\zeta_k}(\a)$ and
$\mu_{\zeta_k}(\b)$ is a finite group if and only if $k \in \{2,3,4,5\}$.
Indeed, the monoid generated by $\zeta_k^{-1} \mu_{\zeta_k}(\a)$ and $\zeta_k^{-2} \mu_{\zeta_k}(\b)$ is isomorphic to the cyclic group~$C_3$ when $k = 2$, the quaternion group~$Q_8$ when $k = 3$, the special linear groups $\mathrm{SL}_2(\mathbb{F}_3)$ and $\mathrm{SL}_2(\mathbb{F}_5)$ when $k = 4$ and $k = 5$ respectively; since $\zeta_k^k = 1$, the corresponding monoids generated by $\mu_{\zeta_k}(\a)$ and $\mu_{\zeta_k}(\b)$ are also finite groups.
For $k = 6$, the monoid is by Lemma~\ref{lem:mu_q_evaluated_at_q6} isomorphic to the abelianization of $\{\a,\b\}^*$, i.e., $(\mathbb{N}^2,+)$. 
For $k \ge 7$, the matrix $\zeta_k^{-1} \mu_{\zeta_k}(\a)$ has an eigenvalue $>1$ since $\zeta_k + \zeta_k^{-1} > 1$ and the characteristic polynomial of $q^{-1} \mu_q(\a)$ is $x^2 - (q{+}1{+}q^{-1}) x + 1$, which implies that the generated monoid is infinite, thus the monoid generated by $\mu_{\zeta_k}(\a)$ is also infinite. 

For $k \in \{2,3,4,5\}$, we also observe the following relations between the residue class of $\mu_1(w)_{12} \pmod{k}$ and $\mu_{\zeta_k}(w)_{12}$ for $w\in\{\a,\b\}^*$ (these relations hold not only for the $12$-coefficient but for all coefficients of $\mu_1(w)$ and $\mu_{\zeta_k}(w)$ and can be verified by induction on the length of~$w$):
\begin{itemize}
    \item $\mu_1(w)_{12} \;(\bmod\;2) \equiv 
        \begin{cases}
            0 \text{ if and only if }\mu_{-1}(w)_{12} = 0,\\
            1 \text{ if and only if }\mu_{-1}(w)_{12} \in\{-1,1\},
        \end{cases}$
    \item $\mu_1(w)_{12} \;(\bmod\;3) \equiv 
        \begin{cases}
            0 \text{ if and only if }\mu_{\zeta_3}(w)_{12} = 0,\\
            1 \text{ if and only if }\mu_{\zeta_3}(w)_{12} \in \{1, \zeta_3, \zeta_3^2\},\\
            2 \text{ if and only if }\mu_{\zeta_3}(w)_{12} \in \{-1, -\zeta_3, -\zeta_3^2\},
        \end{cases}$
    \item $\mu_1(w)_{12} \;(\bmod\;4) \equiv 
        \begin{cases}
            0 \text{ if and only if }\mu_{i}(w)_{12} = 0,\\
            1 \text{ or } 3 
              \text{ if and only if }\mu_{i}(w)_{12} \in \{\pm 1,\pm i\},\\
            2 \text{ if and only if }\mu_{i}(w)_{12} \in \{1\pm i,-1\pm i\}.\\
        \end{cases}$
\end{itemize}
The value $\mu_1(w)_{12} \mod 5$ can also be deduced from $\mu_{\zeta_5}(w)_{12}$,
which takes 31 distinct values in the complex plane, see Figure~\ref{fig:mu_zeta5}.
For $k \ge 6$, we have not found relations between the residue class of $\mu_1(w)_{12} \pmod{k}$ and $\mu_{\zeta_k}(w)_{12}$.

\begin{figure}
\begin{center}
    \includegraphics[width=8cm]{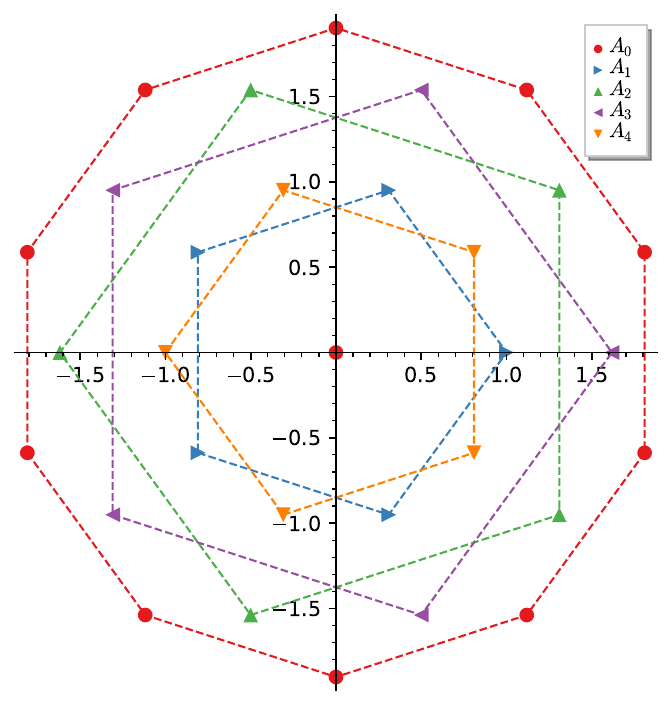}
\end{center}
    \caption{For $w\in\{\a,\b\}^*$, $\mu_{\zeta_5}(w)_{{12}}$ takes 31 different values.
    The set $A_k=\{\mu_{\zeta_5}(w)_{{12}} \mid \mu_1(w)_{{12}} \equiv k \mod 5,\, w\in\{\a,\b\}^*\}$ 
    consists of the vertices of a regular pentagon when $k \in \{1,2,3,4\}$,
    of the vertices of a regular decagon and the origin when $k=0$.}
    \label{fig:mu_zeta5}
\end{figure}
\end{remark}

\section{$w \mapsto \mu_q(w)_{12}$ is not injective on $\{\a,\b\}^*$}\label{sec:muq-not-injective}

In this section, we provide a list of pairs of words over the alphabet
$\{\a,\b\}$ for which $w \mapsto \mu_q(w)_{12}$ is not injective.
For example, $\a\a\a\b\b$ and $\a\b\a\a\b$ have the same image as we have
\begin{align*}
    \mu_q(\a\a\a\b\b)_{12}
    &=1{+}4 q{+}10 q^{2}{+}19 q^{3}{+}27 q^{4}{+}33 q^{5}{+}34 q^{6}{+}29 q^{7}{+}21 q^{8}{+}12 q^{9}{+}5 q^{10}{+}q^{11}\\
    &= \mu_q(\a\b\a\a\b)_{12}.
\end{align*}
The section contains two results:
Theorem~\ref{t:identity1}
and Theorem~\ref{t:identity2}.
All pairs of words we know of are of form
of Equation~\eqref{e:identitymu1} or Equation~\eqref{e:identitymu2}.
So we believe they completely describe the pairs
of words $x,y\in\{\a,\b\}^*$ such that $\mu_q(x)_{12}=\mu_q(y)_{12}$.

\subsection{First result}
To state the results, we need the two involutions $w \mapsto \widetilde{w}$ and $w \mapsto \overline{w}$ on $\{\a,\b\}^*$ which are defined by $\widetilde{w} = w_k \cdots w_1$ and $\overline{w} = \overline{w_k} \cdots \overline{w_1}$ if $w = w_1 \cdots w_k$, with $\overline{\a} = \b$, $\overline{\b} = \a$, i.e., $\widetilde{w}$ is the mirror image of~$w$ and $\overline{w}$ is obtained from $\widetilde{w}$ by exchanging $\a$ and~$\b$.
Also, more generally, we consider images of the homomorphism
\[
M_q:\, \{\a,\b\}^* \to \GL_2(\Z[q^{\pm1}]), \quad \a \mapsto L_q,\ \b \mapsto R_q
\]
which will be used to prove identities for $\mu_q$
since $\mu_q(\a)=M_q(\b\a)$
and   $\mu_q(\b)=M_q(\b\b\a\a)$.

\begin{theorem} \label{t:identity1}
For all $w \in \{\a,\b\}^*$, $k, m, n \ge 0$, we have
\begin{align}
M_q\big(\a^k\b w \b\a^m\big)_{12} & = M_q\big(\a^k\b \overline{w} \b\a^n\big)_{12}, \label{e:identityM1} \\[.5ex]
\mu_q(\a w \b)_{12} & = \mu_q(\a \widetilde{w} \b)_{12}. \label{e:identitymu1}
\end{align}
\end{theorem}

\begin{proof}
We have $M_q(\a^k w \a^m)_{12} = q^k M_q(w)_{12}$ for all $w \in \{\a,\b\}^*$, $k, m \ge 0$, because $(1,0)\, L_q = (q,0)$ and $L_q {}^t\!(1,0) = {}^t\!(1,0)$.
Hence, it suffices to prove \eqref{e:identityM1} for $k = m = n = 0$.
Since 
\[
Q_q\,L_q\,Q_q^{-1} = {}^t\!R_q, \quad \mbox{and} \quad  Q_q\,R_q\,Q_q^{-1} = {}^t\!L_q, \quad \mbox{with}\ Q_q = \begin{pmatrix}q&0\\0&1\end{pmatrix},
\]
we have, for $w = w_1 \cdots w_\ell \in \{\a,\b\}^*$,
\begin{equation} \label{e:bar}
Q_q\,M_q(w)\,Q_q^{-1} = {}^t\!M_q(\overline{w_1}) \cdots {}^t\!M_q(\overline{w_\ell}) = {}^t\!M_q(\overline{w_\ell} \cdots \overline{w_1}) = {}^t\!M_q(\overline{w})
\end{equation}
and thus
\[
\begin{aligned}
M_q(\b w \b)_{12} & = \big(R_q Q_q^{-1} {}^t\!M_q(\overline{w}) Q_q R_q\big)_{12} = (1,1)\, {}^t\!M_q(\overline{w})\, {}^t\!(q,1) = (q,1)\, M_q(\overline{w}) \, {}^t\!(1,1) \\
& = M_q(\b \overline{w} \b)_{12},
\end{aligned}
\]
using that $1 \times 1$ matrices are invariant under transposition.
This proves~\eqref{e:identityM1}.

Let $\sigma:\, \{\a,\b\}^* \to \{\a,\b\}^*$ be the homomorphism given by $\sigma(\a) =\b\a$ and $\sigma(\b) = \b\b\a\a$.
Then we have $\mu_q(w) = M_q(\sigma(w))$ and $\overline{\sigma(w)} = \sigma(\widetilde{w})$ for all $w \in \{\a,\b\}^*$, thus 
\[
\begin{aligned}
\mu_q(\a w \b)_{12} & = M_q\big(\b\a \sigma(w) \b\b\a\a\big)_{12} = M_q\big(\b\a \overline{\sigma(w)} \b\b\a\a\big)_{12} = M_q\big(\b\a \sigma(\widetilde{w}) \b\b\a\a\big)_{12} \\
& = \mu_q(\a \widetilde{w} \b)_{12}, 
\end{aligned}
\]
where we have used \eqref{e:identityM1} and $\overline{\a w \b} = \a \overline{w} \b$ for the second equation.
\end{proof}

Recall that if $\a w\b\in\{\a,\b\}^*$ is a Christoffel word, then $w$ is a palindrome \cite[Theorem 2.3.1]{MR3887697}.
Therefore Theorem~\ref{t:identity1} is compatible with
Theorem~\ref{thm:main}.

\subsection{Second result}

We obtain more identities by images of the homomorphisms (with $w \in \{\a,\b\}^*$)
\[
\begin{aligned}
\varphi_w:\, \{\a,\b,\c,\d\}^* \to \{\a,\b\}^*, \quad 
\a & \mapsto w\a\b\b\a\overline{w}\a\b\b\a, & \c & \mapsto  w\a\b\b\a\overline{w}\b\a\a\b, \\
\b & \mapsto w\b\a\a\b\overline{w}\b\a\a\b, & \d & \mapsto
w\b\a\a\b\overline{w}\a\b\b\a, \\[1ex]
\psi_w:\, \{\a,\b,\c,\d\}^* \to \{\a,\b\}^*, \quad 
\a & \mapsto w\a\b\widetilde{w}\a\b, & \c & \mapsto w\a\b\widetilde{w}\b\a, \\
\b & \mapsto w\b\a\widetilde{w}\b\a, & \d & \mapsto w\b\a\widetilde{w}\a\b,
\end{aligned}
\]
and the involution $w \mapsto \widehat{w}$ on $\{\a,\b,\c,\d\}^*$ defined by $\widehat{w} = \widehat{w_k} \cdots \widehat{w_1}$ if $w = w_1 \cdots w_k$, with $\widehat{\a} = \b$, $\widehat{\b} = \a$, $\widehat{\c} = \c$, $\widehat{\d} = \d$.
Note that $\widehat{w} = \overline{w}$ for $w \in \{\a,\b\}^*$.
For the proof of the following theorem, we also use the extension of $w \mapsto \overline{w}$ to  $\{\a,\b,\c,\d\}^*$ defined by $\overline{\c} = \d$, $\overline{\d} = \c$.

\begin{theorem} \label{t:identity2}
For all $w \in \{\a,\b\}^*$, $v \in \{\a,\b,\c,\d\}^*$, $k, m, n \ge 0$, we have\footnote{In the version of the article that is published in \emph{Algebraic Combinatorics}, the two equations in Theorem~\ref{t:identity2} are erroneously stated with~$\overline{v}$ instead of~$\widehat{v}$.}
\begin{align}
M_q\big(\a^k\b \varphi_w(v) w \b\a^m\big)_{12} & = M_q\big(\a^k\b \varphi_w(\widehat{v}) w \b\a^n\big)_{12}. \label{e:identityM2} \\[.5ex]
\mu_q\big(\a \psi_w(v) w \b\big)_{12} & = \mu_q\big(\a \psi_w(\widehat{v}) w \b\big)_{12}. \label{e:identitymu2}
\end{align}
\end{theorem} 

For the proof of the theorem, we decompose $\varphi_w = \eta_w \circ \tau$ with
\[
\begin{aligned}
\eta_w:\ & \{\a,\b,\c,\d\}^* \to \{\a,\b\}^*, &
\a & \mapsto w\a\b\b\a, & \c & \mapsto \overline{w}\a\b\b\a, \\
& & \b & \mapsto w\b\a\a\b, & \d & \mapsto \overline{w}\b\a\a\b, \\
\tau:\ & \{\a,\b,\c,\d\}^* \to \{\a,\b,\c,\d\}^*, & \a & \mapsto \a\c, & \c & \mapsto \a\d, \\
& & \b & \mapsto \b\d, & \d & \mapsto \b\c,
\end{aligned}
\]
and we use the homomorphism
\[
\begin{aligned}
\eta'_w:\ & \{\a,\b,\c,\d\}^* \to \{\a,\b\}^*, &
\a & \mapsto \a\b\b\a w, & \c & \mapsto \a\b\b\a \overline{w}, \\
& & \b & \mapsto \b\a\a\b w, & \d & \mapsto \b\a\a\b \overline{w}.
\end{aligned}
\]
Then
\[
\varphi_w(\widehat{v}) w = \eta_w(\tau(\widehat{v})) w = w \eta'_w\big(\,\overline{\!\tau(v)\!}\,\big)
\] 
for all $v \in \{\a,\b,\c,\d\}^*$; recall that $\overline{w}$ is defined by reversing~$w$ and exchanging $\a$ and~$\b$ as well as $\c$ and~$\d$, while $\widehat{w}$ only exchanges $\a$ and~$\b$ after reversing~$w$.

We have to show that the difference
\[
\Delta_w(v) = M_q\big(\b \eta_w(v)w \b\big)_{12} - M_q\big(\b w \eta'_w(\overline{v}) \b\big)_{12}
\]
is zero for all 
$v \in \tau(\{\a,\b,\c,\d\}^*) 
       = \{\a\c,\a\d,\b\c,\b\d\}^*
       = (\{\a,\b\}\{\c,\d\})^*$.

\begin{lemma} \label{l:Delta}
Let $a \in \{\c,\d\}$, $v \in (\{\a,\b\}\{\c,\d\})^*$, $w \in \{\a,\b\}^*$. 
If $\Delta_w(v) = 0$, then 
\[
\Delta_w(u \a \overline{u} a v) = \Delta_w(u \b \overline{u} a v)
\]
for all $u \in (\{\a,\b\}\{\c,\d\})^*$ and 
\[
\Delta_w(u \c \overline{u} a v) = \Delta_w(u \d \overline{u} a v)
\]
for all $u \in (\{\a,\b\}\{\c,\d\})^*\{\a,\b\}$.
\end{lemma}

\begin{proof}
Assume first that $|u|$ is even. 
Then
\[
\begin{aligned}
& \Delta_w(u \a \bar{u} a v) - \Delta_w(u \b \bar{u} a v) \\
& = M_q\big(\b \eta_w(u \a\bar{u} a v) w \b\big)_{12} -
M_q\big(\b \eta_w(u \b \bar{u} a v) w \b\big)_{12} \\
& \quad + M_q\big(\b w \eta'_w(\bar{v} \bar{a} u \a \bar{u}) \b\big)_{12} - M_q\big(\b w\eta'_w(\bar{v} \bar{a} u \b \bar{u}) \b\big)_{12} \\
& = \big(M_q\big(\b \eta_w(u)w\big) \big(M_q(\a\b\b\a)-M_q(\b\a\a\b)\big) M_q\big(\eta_w(\bar{u} a v)w \b\big)\big)_{12} \\
& \quad + \big(M_q\big(\b w \eta'_w(\bar{v} \bar{a} u)\big) \big(M_q(\a\b\b\a)-M_q(\b\a\a\b)\big) M_q\big(w\eta'_w(\bar{u}) \b\big)\big)_{12} \\
& = (q^3+1)\, \big(M_q\big(\b \eta_w(u)w\big) S Q_q M_q\big(\eta_w(\bar{u} a v)w \b\big)\big)_{12} \\
& \quad + (q^3+1)\, \big(M_q\big(\b w \eta'_w(\bar{v} \bar{a} u)\big) S Q_q M_q\big(w \eta'_w(\bar{u}) \b\big)\big)_{12} \\
& = (q^3+1)\, q^{|\eta_w(u)w|} \big(R_q S Q_q M_q\big(\overline{w}^{-1} \eta_w(a v) w \b\big)\big)_{12} \\
& \quad + (q^3+1)\, q^{|\eta_w(u)w|} \big(M_q\big(\b w \eta'_w(\bar{v} \bar{a}) \overline{w}^{-1}\big) S Q_q R_q\big)_{12} \\
& = (q^3+1)\, q^{|\eta_w(u)w|} d_a\, \big(M_q\big(\b \eta_w(v) w \b\big)_{12} - M_q\big(\b w \eta'_w(\bar{v}) \b\big)_{12}\big) \\
& = (q^3+1)\, q^{|\eta_w(u)w|} d_a\, \Delta_w(v),
\end{aligned}
\]
with $d_{\c} = -q$ and $d_{\d} = q^4$.
Here, we use for the third equation that 
\[
M_q(\a\b\b\a) = M_q(\b\a\a\b) + (q^3{+}1)\,S Q_q, \quad \mbox{where} \quad S = \begin{pmatrix}0&-1\\1&0\end{pmatrix}, \ Q_q = \begin{pmatrix}q&0\\0&1\end{pmatrix}.
\]
For the fourth equation, we use that, by~\eqref{e:bar},
\[
M_q(z)\, S\, Q_q\, M_q(\overline{z}) =  M_q(z)\, S\, {}^t\!M_q(z)\, Q_q = \det(M_q(z))\, S Q_q = q^{|z|} S Q_q
\]
for all $z \in \{\a,\b\}^*$, in particular for $z = \eta_w(u) w$ (with $\overline{z} = \eta_w(\overline{u}) \overline{w}$) and for $z = \overline{w} \eta'_w(u)$ (with $\overline{z} = w \eta'_w(\overline{u})$).
For the fifth equation, we use that 
\[
\begin{aligned}
(1,0)\, R_q S Q_q M_q(\overline{w}^{-1} \eta_w(\c)) & = (1,0)\, R_q S Q_q M_q(\a\b\b\a) = - (q^2,q) = -q\, (1,0)\, R_q, \\
(1,0)\, R_q S Q_q M_q(\overline{w}^{-1} \eta_w(\d)) & =  (1,0)\, R_q S Q_q M_q(\b\a\a\b) = (q^5,q^4) = q^4 (1,0)\, M_q(\b), \\
M_q(\eta'_w(\d) \overline{w}^{-1}\big) S Q_q R_q {}^t\!(0,1) & = M_q(\b\a\a\b) S Q_q R_q {}^t\!(0,1) = {}^t\!(q,q) = q\, M_q(\b)\, {}^t\!(0,1), \\
M_q(\eta'_w(\c) \overline{w}^{-1}\big) S Q_q R_q {}^t\!(0,1) & = M_q(\a\b\b\a) S Q_q R_q {}^t\!(0,1) = {-} {}^t\!(q^4,q^4) = {-} q^4 R_q {}^t\!(0,1).
\end{aligned}
\]
Therefore, $\Delta_w(v) = 0$ implies that $\Delta_w(u \a \bar{u} a v) = \Delta_w(u \b \bar{u} a v)$.

The proof of $\Delta_w(u \c \bar{u} a v) = \Delta_w(u \d \bar{u} a v)$ for odd $|u|$ runs along the same lines.
\end{proof}

\begin{lemma} \label{l:Delta2}
For all $v \in (\{\a,\b\}\{\c,\d\})^*$, $w \in \{\a,\b\}^*$, we have $\Delta_w(v) = 0$. \end{lemma}

\begin{proof}
We proceed by induction on the length of~$v$. 
The statement is trivially true for $|v| = 0$. 
Suppose that it is true up to length $k-1$ and consider it for length~$k$.

We claim that the value of $\Delta_w(v_1 \cdots v_{2k})$ does not depend on the choice of $v_1 \cdots v_j$, for any $j \le k$. 
The claim is true for $j = 1$, by Lemma~\ref{l:Delta} with $u = \varepsilon$ and the induction hypothesis.
If the claim is true up to $j{-}1$, then it gives together with Lemma~\ref{l:Delta}, for any $u_1 \cdots u_j \in (\{\a,\b\}\{\c,\d\})^* \cup (\{\a,\b\}\{\c,\d\})^* \{\a,\b\}$, that
\[
\begin{split}
& \Delta_w(u_1 \cdots u_j v_{j+1} \cdots v_{2k}) = \Delta_w(\overline{v_{j+1} \cdots v_{2j-1}} u_j v_{j+1} \cdots v_{2k}) \\
& \qquad = \Delta_w(\overline{v_{j+1} \cdots v_{2j-1}} v_j v_{j+1} \cdots v_{2k}) = \Delta_w(v_1 \cdots v_{2k}).
\end{split}
\]
This proves the claim. 

Since $\eta_w(\overline{u}u) w = w \eta'_w(u\overline{u})$ for all $u \in (\{\a,\b\}\{\c,\d\})^* \cup (\{\a,\b\}\{\c,\d\})^* \{\a,\b\}$, we have $\Delta_w(\overline{v_{k+1} \cdots v_{2k}}v_{k+1} \cdots v_{2k}) = 0$, thus $\Delta_w(v_1 \cdots v_{2k}) = 0$ for all $v_1 \cdots v_{2k} \in (\{\a,\b\}\{\c,\d\})^*$. 
\end{proof}

\begin{proof}[Proof of Theorem~\ref{t:identity2}]
As for~\eqref{e:identityM1}, it suffices to prove \eqref{e:identityM2} for $k = m = n = 0$.
Since $\varphi_w(v) = \eta_w(\tau(v))$ and $\varphi_w(\overline{v}) w = w \eta'_w\big(\,\overline{\!\tau(v)\!}\,\big)$ for all $w \in \{\a,\b\}^*$, $v \in \{\a,\b,\c,\d\}^*$, Lemma~\ref{l:Delta2} implies that \eqref{e:identityM2} holds.

Let $\sigma$ be as in the proof of Theorem~\ref{t:identity1}.
Then 
\[
\begin{aligned}
\mu_q\big(\a \psi_w(v) w \b\big)_{12} & = M_q(\b \eta_{\a\sigma(w)\b}(v) \a \sigma(w) \b\b\a\a) = M_q(\b \eta_{\a\sigma(w)\b}(\widehat{v}) \a \sigma(w) \b\b\a\a) \\
& = \mu_q\big(\a \psi_w(\widehat{v}) w \b\big)_{12},
\end{aligned}
\]
using that $\a \sigma(\psi_w(v)) = \eta_{\a\sigma(w)\b}(v) \a$, and using \eqref{e:identityM2} for the second equation. 
\end{proof}

The equation $M_q(x)_{12}=M_q(y)_{12}$ has many solutions
$x,y\in\{\a,\b\}^*$ 
which are not of the form
of Equation~\eqref{e:identityM1} or~\eqref{e:identityM2}, for example 
\[
M_q(\b\b\a\a\a\a\a\b\b)_{12} = 1 + 2q + 3q^2 + 4q^3 + 4q^4 + 4q^5 + 3q^6 + 2q^7 + q^8 = M_q(\b\a\a\a\b\a\a\a\b)_{12},
\]
but we believe that Equations~\eqref{e:identitymu1} and~\eqref{e:identitymu2} are complete.

\begin{question}
    Do there exist $x,y\in\{\a,\b\}^*$ satisfying
    $\mu_q(x)_{12}=\mu_q(y)_{12}$
    which are not given by Equation~\eqref{e:identitymu1} or~\eqref{e:identitymu2}?
\end{question}

\bibliographystyle{myalpha} 
\bibliography{biblio}

\end{document}